\documentclass[preprint,12pt,3p]{elsarticle}

\usepackage{color}

\usepackage{amsthm}
\usepackage{amsmath}
\usepackage[utf8]{inputenc}
\usepackage[english]{babel}
\usepackage{amsfonts}
\usepackage{graphicx}
\usepackage{epstopdf}
\usepackage{algorithm}
\usepackage{algorithmic}
\usepackage[dvipsnames]{xcolor}

\newcommand{\R}{\mathbb{R}}

\usepackage{amsmath,amssymb}
\DeclareMathOperator{\E}{\mathbb{E}}

\DeclareMathOperator*{\argmin}{arg\,min}

\usepackage{mathtools}
\DeclarePairedDelimiterX{\norm}[1]{\lVert}{\rVert}{#1}
\DeclarePairedDelimiterX{\inp}[2]{\langle}{\rangle}{#1, #2}

\usepackage[symbol]{footmisc}

\usepackage{graphicx}

\newtheorem{theorem}{Theorem}
\newtheorem{lemma}[theorem]{Lemma}
\newdefinition{remark}{Remark}

\newproof{pf}{Proof}
\newproof{pot}{Proof of Theorem \ref{thm2}}
\theoremstyle{plain}

\usepackage[dvipsnames]{xcolor}

\usepackage[symbol]{footmisc}

\biboptions{sort&compress}

\usepackage{lipsum}
\makeatletter
\def\ps@pprintTitle{
 \let\@oddhead\@empty
 \let\@evenhead\@empty
 \def\@oddfoot{}
 \let\@evenfoot\@oddfoot}
\makeatother

\begin{document}

\begin{frontmatter}

\title{Penalty \& Augmented Kaczmarz Methods For Linear Systems \& Linear Feasibility Problems}

\author[label1]{Md Sarowar Morshed}
\address[label1]{Independent Researcher}
\ead{riponsarowar@outlook.com}

\begin{abstract}
In this work, we shed light on the so-called Kaczmarz method for solving \textit{Linear System} (LS) and \textit{Linear Feasibility} (LF) problems from a optimization point of view. We introduce well-known optimization approaches such as
Lagrangian penalty and Augmented Lagrangian in the \textit{Randomized Kaczmarz} (RK) method. In doing so, we propose two variants of the RK method namely the \textit{Randomized Penalty Kacmarz} (RPK) method and \textit{Randomized Augmented Kacmarz} (RAK) method. We carry out convergence analysis of the proposed methods and obtain linear convergence results. 
\end{abstract}

\begin{keyword}
Randomized Kaczmarz \sep Linear System \sep Linear Feasibility \sep Augmented Lagrangian \sep Method of Multipliers \sep Penalty Method
\MSC[2010]
90C05\sep 90C51 \sep 65K05 \sep 65B05 \sep 34A25

\end{keyword}

\end{frontmatter}



\section{Introduction} \label{sec:intro}
In this work, we concentrated on solving the following fundamental problems:
\begin{align}
    & \textbf{LS:} \quad Ax = b, \ \ b \in \R^m, \ A \in \R^{m\times n} \label{ls} \\
    & \textbf{LF:} \quad Ax \leq b, \ \ b \in \R^m, \ A \in \R^{m\times n} \label{lf}
\end{align}
Problems \eqref{ls} and \eqref{lf} are significant and central to a wide variety of computational fields such as \textit{Numerical Linear Algebra},  \textit{Scientific Computing}, \textit{Convex Optimization}, \textit{Signal Processing}, etc. For instance solving LS problems approximately is of practical benefit in the inexact Newton schemes for solving large-scale optimization problems. With the advent of big data, iterative methods such as Kaczmarz-type methods are gaining considerable attraction among researchers for their simplicity and efficiency. Kaczmarz method was discovered by Kaczmarz \cite{kaczmarz:1937} in 1937. Variant of Kaczmarz methods gained traction after the seminal works of Strohmer \textit{et. al} \cite{strohmer:2008} and Leventhal \textit{et. al} \cite{lewis:2010}. Recently, Kaczmarz-type methods have been explored to solve a wide variety of problems such as LS, LF, least square, low-rank matrix recovery, etc \citep{needell:2010,eldar:2011,zouzias:2013,NEEDELL:2014,ma:2015,NEEDELL:2015, blockneddel:2015,needell:2016, quadratic:2016,needell:2016,haddock:2017,haddock:2019, razaviyayn:2019, haddock:2019,needell2019block, Morshed2019, morshed2020generalization, morshed:momentum, morshed:sketching, morshed2020stochastic}.

\paragraph{Kaczmarz Method for LS $\& $ LF} Starting from any random point $x_k$, Kaczmarz method updates the next iterate $x_{k+1}$ by solving the following optimization problems:
\begin{align}
& \textbf{LS:} \quad x_{k+1} =     \argmin_x \frac{1}{2}\|x-x_k\|^2 \quad \textbf{s.t} \quad  a_i^Tx = b_i, \label{mp1} \\
& \textbf{LF:} \quad x_{k+1} =     \argmin_x \frac{1}{2}\|x-x_k\|^2 \quad \textbf{s.t} \quad  a_i^Tx \leq b_i, \label{mp1000}
\end{align}
We can derive the closed form solutions of the above problems as follows:
\begin{align}
\label{mp2}
& \textbf{LS:} \quad x_{k+1} = x_k - \frac{a_i^Tx_k-b_i}{\|a_i\|^2} \ a_i, \quad \textbf{LF:} \quad x_{k+1} = x_k - \frac{\left(a_i^Tx_k-b_i\right)^+}{\|a_i\|^2} \ a_i
\end{align}

\section{Lagrangian Penalty Approaches}
\label{sec:lp}
In this section, we first discuss the Kaczmarz method for solving both LS and LF problems. Then, we introduce the proposed approaches based on Lagrangian multiplier.

\paragraph{Penalty Kaczmarz Method} Let's write the following penalty function formulation of the above optimization problems. Introducing the penalty parameter $\rho > 0$ in the formulations \eqref{mp1} and \eqref{mp1000}, we get the following:
\begin{align}
& \textbf{LS:} \quad  x_{k+1} =     \argmin_x \mathcal{L}_1(x, \rho) =  \frac{1}{2} \|x-x_k\|^2 + \frac{\rho}{2}  \ | a_i^Tx-b_i|^2. \label{mp3} \\
& \textbf{LF:} \quad  x_{k+1} =     \argmin_x \mathcal{L}_2(x, \rho) =  \frac{1}{2} \|x-x_k\|^2 + \frac{\rho}{2}  \ | \left(a_i^Tx-b_i\right)^+|^2. \label{mp4}
\end{align}
Setting $\frac{\partial \mathcal{L}_1(x, \rho)}{\partial x} = 0$, we get
\begin{align}
\label{mp5}
    0 = x-x_k +  \rho  \  a_i (a_i^Tx-b_i) \Rightarrow  x = \left(I + \rho a_i a_i^T \right)^{-1} \left[x_k + \rho \  b_i a_i\right] = x_k - \frac{ a_i^T x_k -b_i  }{\frac{1}{\rho}+ \|a_i\|^2} \ a_i
\end{align}
Here, we used the Sherman–Morrison matrix identity, $\left(I + \rho a_i a_i^T\right)^{-1} = I  - \frac{a_i a_i^T}{\frac{1}{\rho}+ \|a_i\|^2}$. The solution is unique as the matrix $I + \rho a_i a_i^T$ is invertible for any $\rho > 0$. Now, for the problem \eqref{mp4}, we only need to consider the case $a_i^Tx - b_i \geq 0$. This is because whenever $ a_i^Tx - b_i  < 0$, we have $x_{k+1} = x_k$. Setting $\frac{\partial \mathcal{L}_2(x, \rho)}{\partial x} = 0$, we get
\begin{align*}
    0 = x-x_k +  \rho  \  a_i (a_i^Tx-b_i) \Rightarrow  x = \left(I + \rho a_i a_i^T \right)^{-1} \left[x_k + \rho \  b_i a_i\right]
\end{align*}
Simplifying like before, we get the following relations:
\begin{align}
\label{mp6}
a_i^Tx_{k+1} - b_i \geq 0: \quad x_{k+1} =  x_k - \frac{ a_i^T x_k -b_i  }{\frac{1}{\rho}+ \|a_i\|^2} \ a_i, \quad \quad a_i^Tx_{k+1} - b_i < 0: \quad x_{k+1} =  x_k 
\end{align}
Now note that the condition $a_i^Tx_{k+1} - b_i \geq 0$ implies $\frac{(a_i^Tx_{k} - b_i) \|a_i\|^2}{\frac{1}{\rho}+ \|a_i\|^2} \geq 0$. This is possible whenever we have $a_i^Tx_{k} - b_i \geq 0$. Therefore, we can write the above relations as follows:
\begin{align}
\label{mp7}
a_i^Tx_{k} - b_i \geq 0: \quad x_{k+1} =  x_k - \frac{ a_i^T x_k -b_i  }{\frac{1}{\rho}+ \|a_i\|^2} \ a_i, \quad \quad a_i^Tx_{k} - b_i < 0: \quad x_{k+1} =  x_k 
\end{align}
Combining the above together, we get the following relation:
\begin{align}
\label{mp8}
x_{k+1} =  x_k - \frac{ \left(a_i^T x_k -b_i\right)^+  }{\frac{1}{\rho}+ \|a_i\|^2} \ a_i
\end{align}

 \begin{algorithm}
\caption{$x_{K+1} = \textbf{RPK}(A,b, c, K)$}
\label{alg:spl}
\begin{algorithmic}
\STATE{Choose initial point $x_0 \in \R^n,  \ \rho_0 \in \R$}
\WHILE{$k \leq K$}
\STATE{Select index $i$ with probability $p_i = \|a_i\|^2/\|A\|^2_F$ and update
\begin{equation*}
 r_k = \begin{cases}
a_i^T x_k -b_i , \ \ \textbf{LS} \\
\left(a_i^T x_k -b_i\right)^+ , \ \ \textbf{LF}
\end{cases} , \quad  x_{k+1} =  x_k - \frac{ r_k \ a_i }{\frac{1}{\rho_k}+ \|a_i\|^2} , \quad \rho_{k+1} = c \rho_k.
\end{equation*}
$k \leftarrow k+1$;}
\ENDWHILE
\end{algorithmic}
\end{algorithm}

\subsection{Randomized Augmented Kaczmarz (RAK)}
\label{sec:als}
Now, we discuss the proposed RAK method for solving both the LS and LF problems. Using augmented Lagrangian penalty function, we car rewrite problems \eqref{mp1} and \eqref{mp1000} as follows:
\begin{align}
\label{mp9}
& \textbf{LS:} \quad x_{k+1} =   \argmin_{x} \frac{1}{2} \|x-x_k\|^2 + \frac{\rho}{2} |a_i^Tx-b_i|^2 \ \ \ \textbf{s.t}. \ \ \ a_i^Tx = b_i \\
& \textbf{LF:} \quad x_{k+1} =   \argmin_{x} \frac{1}{2} \|x-x_k\|^2 + \frac{\rho}{2} |(a_i^Tx-b_i)^+|^2 \ \ \ \textbf{s.t}. \ \ \ a_i^Tx \leq b_i
\end{align}
The Lagrangians of the above problems can be simplified as follows:
\begin{align}
\label{mp10}
 & \textbf{LS:} \quad   \mathcal{L}_1 (x,z, \rho) = \frac{1}{2} \|x-x_k\|^2  +  \frac{\rho}{2} |a_i^Tx-b_i|^2  + z  (a_i^Tx-b_i) \\
 & \textbf{LF:} \quad  \mathcal{L}_2 (x,z, \rho) = \frac{1}{2} \|x-x_k\|^2 + \frac{1}{2 \rho} \big | \left\{z+ \rho (a_i^Tx-b_i) \right \}^+ \big |^2
\end{align}
Then, we formulated the Augmented Kaczmarz scheme by setting up the following update formulas:
\begin{align}
   & \textbf{LS:} \quad  x_{k+1} = \argmin_{x} \mathcal{L}_1 (x,z_k, \rho), \quad  z_{k+1} = z_k + \rho (a_i^Tx_{k+1}-b_i) \label{mp11} \\
    & \textbf{LF:} \quad x_{k+1} = \argmin_{x} \mathcal{L}_2 (x,z_k, \rho), \quad z_{k+1} = \left\{z_k+ \rho (a_i^Tx_{k+1}-b_i) \right \}^+ \label{mp12}
\end{align}
We can explicitly derive the update formulas in closed form. Setting $\frac{\partial \mathcal{L} (x,z_k, \rho)}{\partial x} = 0$, we get
\begin{align*}
    0 = x-x_k+  z_k  a_i + \rho  a_i (a_i^T x - b_i) \Rightarrow  x = \left(I + \rho a_i a_i^T \right)^{-1} \left[x_k + \rho \  b_i a_i - z_k  a_i \right]
\end{align*}
The solution is unique as the matrix $I + \rho a_i a_i^T$ is invertible for any $\rho > 0$. Using the Sherman–Morrison matrix identity we get the following simplified identity:
\begin{align}
\label{mp5}
x_{k+1} =  x_k - \frac{ a_i^T x_k -b_i + \frac{1}{\rho} z_k }{\frac{1}{\rho}+ \|a_i\|^2} \ a_i, \quad \Rightarrow z_{k+1} = z_k + \rho (a_i^Tx_{k+1}-b_i) = \frac{ a_i^T x_k -b_i + \frac{1}{\rho} z_k }{\frac{1}{\rho}+ \|a_i\|^2} 
\end{align}
Now, for the problem \eqref{mp12}, we only need to consider the case $ z_k+ \rho (a_i^Tx - b_i) \geq 0$. This is because whenever $ z_k+ \rho (a_i^Tx - b_i) < 0$, we have $x_{k+1} = x_k, \ z_{k+1} = 0$. Setting $\frac{\partial \mathcal{L}_2 (x,z_k, \rho)}{\partial x} = 0$, we get
\begin{align*}
    0 = x-x_k+  z_k  a_i + \rho  a_i (a_i^T x - b_i) \Rightarrow  x = \left(I + \rho a_i a_i^T \right)^{-1} \left[x_k + \rho \  b_i a_i - z_k  a_i \right]
\end{align*}
Simplifying like before, we get the following simplified relations:
\begin{align*}
& z_k+ \rho (a_i^Tx_{k+1} - b_i) \geq 0: \quad x_{k+1} = x_k - \frac{ a_i^T x_k -b_i + \frac{1}{\rho} z_k }{\frac{1}{\rho}+ \|a_i\|^2} \ a_i , \quad   z_{k+1} = \frac{ a_i^T x_k -b_i + \frac{1}{\rho} z_k }{\frac{1}{\rho}+ \|a_i\|^2}  \\
& z_k+ \rho (a_i^Tx_{k+1} - b_i) < 0: \quad x_{k+1} =  x_k, \quad z_{k+1} = 0
\end{align*}
Now, note that whenever $z_k+ \rho (a_i^Tx_{k+1} - b_i) < 0$ holds we have $x_{k+1} =  x_k, \quad z_{k+1} = 0$. This implies that $z_k+ \rho (a_i^Tx_{k+1} - b_i) = z_k+ \rho (a_i^Tx_{k} - b_i)$. Similarly, we will show that the condition $z_k+ \rho (a_i^Tx_{k+1} - b_i) \geq 0$ is equivalent to the condition $z_k+ \rho (a_i^Tx_{k} - b_i) \geq 0$ for the first case. To prove that we use the expression of $x_{k+1}$ and get the following:
\begin{align*}
z_k+  \rho (a_i^Tx_{k+1} - b_i)  & = z_k+ \rho (a_i^Tx_{k} - b_i) -   \frac{ \rho \left[a_i^Tx_{k} - b_i + \frac{1}{\rho} z_k\right]}{\frac{1}{\rho}+ \|a_i\|^2}  \|a_i\|^2 = \frac{ \frac{1}{\rho} z_k+ a_i^Tx_{k} - b_i}{\frac{1}{\rho}+ \|a_i\|^2}
\end{align*}
as $\rho > 0$, the term $\frac{1}{\rho}+ \|a_i\|^2$ is always greater than zero. This proves the equivalency. Using this condition, we have the following: 
\begin{align*}
& z_k+ \rho (a_i^Tx_{k} - b_i) \geq 0: \quad x_{k+1} = x_k - \frac{ a_i^T x_k -b_i + \frac{1}{\rho} z_k }{\frac{1}{\rho}+ \|a_i\|^2} \ a_i , \quad   z_{k+1} = \frac{ a_i^T x_k -b_i + \frac{1}{\rho} z_k }{\frac{1}{\rho}+ \|a_i\|^2}  \\
& z_k+ \rho (a_i^Tx_{k} - b_i) < 0: \quad x_{k+1} =  x_k, \quad z_{k+1} = 0
\end{align*}
Combining the above together, we get the following relation:
\begin{align}
x_{k+1} =  x_k - \frac{ \left(a_i^T x_k -b_i + \frac{1}{\rho} z_k \right)^+  }{\frac{1}{\rho}+ \|a_i\|^2} \ a_i, \quad z_{k+1} =  \frac{ \left(a_i^T x_k -b_i + \frac{1}{\rho} z_k \right)^+  }{\frac{1}{\rho}+ \|a_i\|^2}
\end{align}

 \begin{algorithm}
\caption{$x_{K+1} = \textbf{RAK}(A,b, c, K)$}
\label{alg:ssd}
\begin{algorithmic}
\STATE{Choose initial points $x_0 \in \R^n, \ z_0 \in \R^p, \ \rho_0 \in \R$}
\WHILE{$k \leq K$}
\STATE{Select index $i$ with probability $p_i = \|a_i\|^2/\|A\|^2_F$ and update
\begin{equation*}
 z_{k+1} = \begin{cases}
\frac{  (a_i^T x_k -b_i) + \frac{1}{\rho_k} z_k }{\frac{1}{\rho_k} +\|a_i\|^2} , \ \ \textbf{LS} \\
\frac{ \left[  (a_i^T x_k -b_i) + \frac{1}{\rho_k} z_k \right]^+ }{\frac{1}{\rho_k} +\|a_i\|^2} , \ \ \textbf{LF}
\end{cases} , \quad  x_{k+1} =  x_k - z_{k+1} a_i , \quad \rho_{k+1} = c \rho_k.
\end{equation*}
$k \leftarrow k+1$;}
\ENDWHILE
\end{algorithmic}
\end{algorithm}

\begin{remark}
\label{rem:1}
Note that, in Algorithm \ref{alg:ssd} we provide an adaptive version of the RAK method. In the adaptive version, the penalty parameter $\rho$ is updated gradually in a way such that $\rho_{k}$ become very large after finitely many iterations. To achieve this, we set a step size $c>1$ and update the penalty parameter as follows: 
\begin{align*}
   & \textbf{LS:} \quad  x_{k+1} = \argmin_{x} \mathcal{L}_1 (x,z_k, \rho_k), \quad z_{k+1} = z_k + \rho_k  (a_i^Tx_{k+1}-b_i ),  \quad \rho_{k+1} = c \rho_k \\
   & \textbf{LF:} \quad  x_{k+1} = \argmin_{x} \mathcal{L}_2 (x,z_k, \rho_k), \quad z_{k+1} = \left[z_k + \rho_k  (a_i^Tx_{k+1}-b_i )\right]^+,  \quad \rho_{k+1} = c \rho_k
\end{align*}
This scheme improves the efficiency of the proposed RAK method. This type of iterative approach is frequently used in the classical augmented Lagrangian framework \cite{Hestenes1969,BERTSEKAS}. Furthermore, if we take $\rho \rightarrow \infty$ in the above relations, we get the update formulas of the randomized Kaczmarz method. To see this, first note that, for the LS and LF problems taking $\rho \rightarrow \infty$, we get
\begin{align}
  & \textbf{LS:} \quad  x_{k+1} =  x_k - \frac{ a_i^T x_k -b_i  }{ \|a_i\|^2} \ a_i, \quad \textbf{LF:} \quad x_{k+1} =  x_k - \frac{ \left(a_i^T x_k -b_i \right)^+  }{ \|a_i\|^2} \ a_i
\end{align}
This is precisely the RK update formulas for the LS and LF problems. For the dual sequences in the limit the relations hold trivially. To show this, we note that the following holds:
\begin{align*}
& \textbf{LS:} \quad   0 = \lim_{\rho \rightarrow \infty} \frac{1}{\rho} \left(z_{k+1}- z_{k} \right) = a_i^Tx_{k+1}-b_i  \\
& \textbf{LF:} \quad  0 = \lim_{\rho \rightarrow \infty} \frac{1}{\rho} z_{k+1} = \lim_{\rho \rightarrow \infty} \left(a_i^T x_{k+1} -b_i + \frac{1}{\rho} z_k \right)^+ = \left(a_i^T x_{k+1} -b_i\right)^+
\end{align*}
The above relations mean we have $a_i^Tx_{k+1} = b_i$ and $a_i^Tx_{k+1} \leq b_i$, respectively for the LS and LF problems.
\end{remark}

\section{Convergence Theory}  We now present the convergence results for the proposed RPK and RAK methods. We will show that for the LS problem $\E[\|x_{k}-x^*\|^2] \rightarrow 0$ where, $ x^*  = \argmin_{Ax = b} \|x_0-x\|^2  = x_0 - A^{\dagger} (Ax_0-b) $. And for the LF problem, $\E[d(x_k, \mathcal{X})^2] \rightarrow 0$, where, $d(x_k, \mathcal{X})^2 =  \argmin_{x \in \mathcal{X}} \|x_k-x\|^2 = \|x_k- \mathcal{P} (x_{k})\|^2$, $\mathcal{P}(x_{k})$ denotes the projection of $x_k$ onto $\mathcal{X}$.

\begin{lemma}
\label{lem0}
(Hoffman \cite{hoffman}, Theorem 4.4 in \cite{lewis:2010}) Let $x \in \R^n$ and $\mathcal{X}$ be the feasible region, then there exists a constant $L > 0$ such that the following identity holds:
\begin{align*}
  d(x,\mathcal{X})^2 \leq L^2 \ \|(Ax-b)^+\|^2.
\end{align*}
\end{lemma}
The constant $L$ is the so-called Hoffman constant. Throughout the convergence analysis, we assume $\|a_i\|^2 = 1$ for all $i$.

\begin{theorem}
\label{th:0}
Consider the RPK algorithm applied to the LS and LF problems. Then, the sequences generated by the RPK algorithm converge and the following hold:
\begin{align*}
   & \textbf{LS:} \quad  \E[\|x_{k+1}-x^*\|^2]  \leq \left(1-\frac{\rho(\rho+2)}{(1+\rho)^2} \frac{\lambda_{\min}(A^TA)}{m}\right)^{k+1} \ \E[\|x_{0}-x^*\|^2]  \\
   & \textbf{LF:} \quad  \E[d(x_{k+1}, \mathcal{X})^2]  \leq \left(1-\frac{\rho(\rho+2)}{(1+\rho)^2} \frac{1}{mL^2}\right)^{k+1} \ \E[d(x_{0}, \mathcal{X})^2]   
\end{align*}
\end{theorem} 

\begin{proof}
Since, $a_i^Tx^* = b_i$ and $\|a_i\|^2 = 1$, using the update formula of the LS problem, we have
\begin{align}
    \|x_{k+1}-x^*\|^2  & = \Big \|x_k-x^* - \frac{\rho(a_i^Tx_k-b_i)}{1+\rho} \ a_i \Big \|^2 = \|x_k-x^*\|^2 - \frac{\rho(\rho+2)}{(1+\rho)^2} \ \left(a_i^Tx_k-b_i\right)^2 \nonumber \\
    & =\|x_k-x^*\|^2 - \frac{\rho(\rho+2)}{(1+\rho)^2} \  (x_k-x^*)^T a_i a_i^T (x_k-x^*)
\end{align}
Taking expectation with respect to index $i$, we get
\begin{align}
   \E_i \left[ \|x_{k+1}-x^*\|^2\right]  & =\|x_k-x^*\|^2 - \frac{\rho(\rho+2)}{(1+\rho)^2} \  (x_k-x^*)^T  \E_i \left[ a_i a_i^T \right] (x_k-x^*) \nonumber \\
   & = \|x_k-x^*\|^2 - \frac{\rho(\rho+2)}{m (1+\rho)^2} \  (x_k-x^*)^T  A^T A (x_k-x^*) \nonumber \\
   & \leq \|x_k-x^*\|^2 - \frac{\rho(\rho+2) \lambda_{\min}(A^TA)}{m (1+\rho)^2} \ \|x_k-x^*\|^2 \nonumber \\
   & =   \left(1-\frac{\rho(\rho+2)}{(1+\rho)^2} \frac{\lambda_{\min}(A^TA)}{m}\right) \ \|x_{k}-x^*\|^2
\end{align}
Here, we used the identity $\E_i \left[ a_i a_i^T \right] = 1/m \sum \nolimits_{i =1}^{m} a_i a_i^T = 1/m \ A^TA$ along with the relation $x^TBx \geq \lambda_{\min}(B) \ \|x\|^2$ for some $B \succeq 0$. Taking expectation again and using the tower property of expectation, we get
\begin{align*}
  \E\left[ \|x_{k+1}-x^*\|^2\right] = \E \left[ \E_i \left[ \|x_{k+1}-x^*\|^2\right] \right] 
   & \leq   \left(1-\frac{\rho(\rho+2)}{(1+\rho)^2} \frac{\lambda_{\min}(A^TA)}{m}\right) \   \E\left[ \|x_{k}-x^*\|^2\right]
\end{align*}
unrolling the recurrence, we get the required result. Similarly for the LF problem, we get
\begin{align}
& d(x_{k+1}, \mathcal{X})^2 = \|x_{k+1}-  \mathcal{P}(x_{k+1})\|^2 \leq \|x_{k+1}-  \mathcal{P}(x_{k})\|^2   = \Big \|x_k-\mathcal{P}(x_{k}) - \frac{\rho(a_i^Tx_k-b_i)^+}{1+\rho} \ a_i \Big \|^2 \nonumber \\
 & = \|x_k-\mathcal{P}(x_{k}) \|^2 + \frac{\rho^2}{(1+\rho)^2} \ |(a_i^Tx_k-b_i)^+|^2 - \frac{2 \rho}{(1+\rho)} (a_i^Tx_k-b_i)^+ [a_i^Tx_k-a_i^T\mathcal{P}(x_{k})] \nonumber \\
 & \leq \|x_k-\mathcal{P}(x_{k}) \|^2 + \frac{\rho^2}{(1+\rho)^2} \ |(a_i^Tx_k-b_i)^+|^2 - \frac{2 \rho}{(1+\rho)} (a_i^Tx_k-b_i)^+ (a_i^Tx_k-b_i) \nonumber \\
 & = d(x_{k}, \mathcal{X})^2 - \frac{\rho(\rho+2)}{(1+\rho)^2}  |(a_i^Tx_k-b_i)^+|^2
\end{align}
here, we used the identity $x^+x = |x^+|^2$. Taking expectation with respect to index $i$, and simplifying we get
\begin{align*}
\E_i \left[d(x_{k+1}, \mathcal{X})^2\right] & \leq d(x_{k}, \mathcal{X})^2 - \frac{\rho(\rho+2)}{(1+\rho)^2}  \E_i \left[ |(a_i^Tx_k-b_i)^+|^2 \right] \nonumber \\
& = d(x_{k}, \mathcal{X})^2 - \frac{\rho(\rho+2)}{m (1+\rho)^2} \sum \limits_{i = 1}^m   |(a_i^Tx_k-b_i)^+|^2 \nonumber \\
& = d(x_{k}, \mathcal{X})^2 - \frac{\rho(\rho+2)}{m (1+\rho)^2}  \| (Ax_k-b)^+ \|^2 \leq \left(1-\frac{\rho(\rho+2)}{(1+\rho)^2} \frac{1}{mL^2}\right) \ d(x_{k}, \mathcal{X})^2  
\end{align*}
Taking expectation again and using the tower property of expectation, we get
\begin{align*}
  \E\left[ d(x_{k+1}, \mathcal{X})^2 \right] = \E \left[ \E_i \left[ d(x_{k+1}, \mathcal{X})^2\right] \right] 
   & \leq   \left(1-\frac{\rho(\rho+2)}{(1+\rho)^2} \frac{1}{mL^2}\right)  \   \E\left[ d(x_{k}, \mathcal{X})^2\right]
\end{align*}
unrolling the recurrence, we get the required result for the LF problem.
\end{proof}

For $k \geq 1$, assume $x_k$ and $z_k$ are random iterates of the RAK algorithm. In the following, we define sequences $\mathcal{V}(x_k,z_k)$ and $\mathcal{U}(x_k,z_k)$.
\begin{align*}
\mathcal{V}(x_k,z_k) = \|x_k-x^*\|^2+ \frac{1}{\rho} \ |z_k |^2, \quad  \mathcal{U}(x_k,z_k) = d(x_{k},\mathcal{X})^2 + \frac{1}{\rho} |z_{k}|^2 
\end{align*}
We will bound the growth of functions $\E[\mathcal{V}(x_k,z_k)]$ and $\E[\mathcal{U}(x_k,z_k)]$. Indeed, we will show that $\E[\mathcal{V}(x_k,z_k)]$ and $\E[\mathcal{U}(x_k,z_k)]$ are Lyapunov functions for LS and LF problems, respectively.

\begin{theorem}
\label{th:1}
Consider the RAK algorithm applied to the LS and LF problems. Then, the sequences generated by the RAK algorithm converge and the following results hold:
\begin{align*}
   & \textbf{LS:} \quad  \E[\mathcal{V}(x_{k+1},z_{k+1})]  \leq \left(1-\frac{\rho}{(1+\rho)} \frac{\lambda_{\min}(A^TA)}{m}\right)^{k+1} \ \E[\mathcal{V}(x_{0},z_{0})]  \\
   & \textbf{LF:} \quad  \E[\mathcal{U}(x_{k+1},z_{k+1})]  \leq  \left(1-\frac{\rho}{(1+\rho)} \frac{1}{mL^2}\right)^{k+1} \ \E[\mathcal{U}(x_{0},z_{0})] 
\end{align*}
\end{theorem}

\begin{proof}
Using the update formula of the LS problem, we have
\begin{align}
 & \mathcal{V}(x_{k+1},z_{k+1}) =   \|x_{k+1}-x^*\|^2 + \frac{1}{\rho} \ |z_{k+1}|^2  =  \|x_k-x^* - z_{k+1} \ a_i \|^2 +  \frac{1}{\rho} \ |z_{k+1}|^2 \nonumber \\
  & = \|x_k-x^*\|^2 -  2 z_{k+1} \ (a_i^Tx_k-b_i) + \frac{1+\rho}{\rho} \ |z_{k+1}|^2  \nonumber \\
    & = \|x_k-x^*\|^2  - \frac{2 (a_i^Tx_k-b_i)}{1+\rho} \left[\rho(a_i^Tx_k-b_i) + z_k \right] + \frac{\rho^2 (a_i^Tx_k-b_i)^2+ |z_k|^2 + 2 \rho z_k  (a_i^Tx_k-b_i) }{\rho(1+\rho)} \nonumber \\
    & = \|x_k-x^*\|^2  - \frac{\rho}{1+\rho} \ (a_i^Tx_k-b_i)^2 +  \frac{1}{\rho}  \  |z_k|^2 - \frac{1}{1+\rho}  \  |z_k|^2
\end{align}
Taking expectation with respect to index $i$, we get
\begin{align}
   \E_i \left[  \mathcal{V}(x_{k+1},z_{k+1}) \right]  & = \mathcal{V}(x_{k},z_{k}) - \frac{\rho}{(1+\rho)} \  (x_k-x^*)^T  \E_i \left[ a_i a_i^T \right] (x_k-x^*) - \frac{1}{1+\rho}  \  |z_k|^2 \nonumber \\
   & \leq \mathcal{V}(x_{k},z_{k}) -  \frac{\rho \ \lambda_{\min}(A^TA)}{m (1+\rho)} \ \|x_k-x^*\|^2 - \frac{1}{1+\rho}  \  |z_k|^2  \nonumber \\
   & \leq   \left(1-\frac{\rho}{(1+\rho)} \frac{\lambda_{\min}(A^TA)}{m}\right) \ \mathcal{V}(x_{k},z_{k}) 
\end{align}
Taking expectation again and using the tower property of expectation, we get
\begin{align*}
  \E\left[  \mathcal{V}(x_{k+1},z_{k+1}) \right] = \E \left[ \E_i \left[  \mathcal{V}(x_{k+1},z_{k+1}) \right] \right] 
   & \leq     \left(1-\frac{\rho}{(1+\rho)} \frac{\lambda_{\min}(A^TA)}{m}\right) \ \E\left[\mathcal{V}(x_{k},z_{k})\right]
\end{align*}
unrolling the recurrence, we get the required result. Using the update formula of the LF problem, we have
\begin{align}
 & \mathcal{U}(x_{k+1},z_{k+1}) =   d(x_{k+1}, \mathcal{X})^2 + \frac{1}{\rho} \ |z_{k+1}|^2  \leq \|x_{k+1} - \mathcal{P}(x_k)\|^2 + \frac{1}{\rho} \ |z_{k+1}|^2 \nonumber \\
 & =  \|x_k-\mathcal{P}(x_k) - z_{k+1} \ a_i \|^2 + \frac{1}{\rho} \ |z_{k+1}|^2 \nonumber \\
  & = \|x_k-\mathcal{P}(x_k)\|^2 -  2 z_{k+1} \ [a_i^Tx_k-a_i^T\mathcal{P}(x_k)] + \frac{1+\rho}{\rho} \ |z_{k+1}|^2  \nonumber \\
    & \leq  d(x_{k}, \mathcal{X})^2   - \frac{2 (a_i^Tx_k-b_i)}{1+\rho} \left[\rho(a_i^Tx_k-b_i) + z_k \right]^+ + \frac{ \big | \left[\rho(a_i^Tx_k-b_i) + z_k \right]^+ \big |^2}{\rho(1+\rho)} \nonumber \\
    & =  d(x_{k}, \mathcal{X})^2  +  \frac{  \left[\rho(a_i^Tx_k-b_i) + z_k \right]^+ }{\rho(1+\rho)} \left[\left[\rho(a_i^Tx_k-b_i) + z_k \right]^+ -2 \rho (a_i^Tx_k-b_i)\right] 
\end{align}
Define, $\mathcal{I}_k = \{i \ | \ \rho(a_i^Tx_k-b_i) + z_k \geq 0 \}$ and $\mathcal{J}_k = \{i \ | \ a_i^Tx_k-b_i \geq 0 \}$ . As $z_k \geq 0$, we have $\mathcal{J}_k  \subseteq \mathcal{I}_k $.  Now, taking expectation with respect to index $i$, we get
\begin{align*}
   \E_i & \left[  \mathcal{U}(x_{k+1},z_{k+1}) \right] \nonumber \\
   & \leq  d(x_{k}, \mathcal{X})^2  +  \sum \limits_{i=1}^m \frac{  \left[\rho(a_i^Tx_k-b_i) + z_k \right]^+ }{m\rho(1+\rho)} \left[\left[\rho(a_i^Tx_k-b_i) + z_k \right]^+ -2 \rho (a_i^Tx_k-b_i)\right]   \nonumber \\
    & =  d(x_{k}, \mathcal{X})^2  +  \sum \limits_{i \in \mathcal{I}_k } \frac{  \left[\rho(a_i^Tx_k-b_i) + z_k \right] }{m\rho(1+\rho)} \left[\left[\rho(a_i^Tx_k-b_i) + z_k \right] -2 \rho (a_i^Tx_k-b_i)\right]   \nonumber \\
    & =  d(x_{k}, \mathcal{X})^2  +  \sum \limits_{i \in \mathcal{I}_k } \frac{ |z_k|^2- \rho^2 (a_i^Tx_k-b_i)^2 }{m\rho(1+\rho)}  \leq  d(x_{k}, \mathcal{X})^2  +\frac{|\mathcal{I}_k| \ |z_k|^2 }{m\rho(1+\rho)} -  \sum \limits_{i \in \mathcal{J}_k }  \frac{\rho (a_i^Tx_k-b_i)^2}{m(1+\rho)}    \nonumber \\
    &  \leq d(x_{k}, \mathcal{X})^2  +\frac{ |z_k|^2 }{\rho(1+\rho)} -  \sum \limits_{i= 1 }^m  \frac{\rho |(a_i^Tx_k-b_i)^+|^2}{m(1+\rho)} =  \mathcal{U}(x_{k},z_{k}) -  \frac{ |z_k|^2 }{1+\rho}-  \frac{\rho \ \|(Ax_k-b)^+\|^2}{m(1+\rho)}    \nonumber \\
   & \leq \mathcal{U}(x_{k},z_{k}) -  \frac{\rho }{m L^2 (1+\rho)} \ d(x_{k}, \mathcal{X})^2 - \frac{1}{1+\rho}  \  |z_k|^2  \leq    \left(1-\frac{\rho}{(1+\rho)} \frac{1}{mL^2}\right) \ \mathcal{U}(x_{k},z_{k}) 
\end{align*}
Taking expectation again and using the tower property of expectation, we get
\begin{align*}
  \E\left[  \mathcal{U}(x_{k+1},z_{k+1}) \right] = \E \left[ \E_i \left[  \mathcal{U}(x_{k+1},z_{k+1}) \right] \right] 
   & \leq     \left(1-\frac{\rho}{(1+\rho)} \frac{1}{mL^2}\right) \ \E\left[\mathcal{U}(x_{k},z_{k})\right]
\end{align*}
unrolling the recurrence, we get the required result.
\end{proof}

\begin{remark}
In the above proof, we provided the convergence analysis considering simplified versions of RAK algorithm. We proved convergence with fixed $\rho$. The proof follows a same argument for the case of adaptive penalty parameter, i.e., $\rho_k$. Indeed, considering () and (), we have the following
\begin{align}
  &  \E_i \left[\mathcal{V}(x_{k+1},z_{k+1})\right]   \leq \mathcal{V}(x_{k},z_{k}) -  \frac{\rho_k \ \lambda_{\min}(A^TA)}{m (1+\rho_k)} \ \|x_k-x^*\|^2 - \frac{|z_k|^2}{1+\rho_k}     -  \frac{c-1}{c \rho_k} |z_{k+1}|^2 \nonumber \\
  &  \E_i \left[  \mathcal{U}(x_{k+1},z_{k+1}) \right]  \leq \mathcal{U}(x_{k},z_{k}) -  \frac{\rho_k  }{m L^2 (1+\rho_k)} \ d(x_{k}, \mathcal{X})^2 - \frac{  |z_k|^2 }{1+\rho_k}    -  \frac{c-1}{c \rho_k} |z_{k+1}|^2 \nonumber
\end{align}
here, we take,  $\mathcal{V}(x_{k},z_{k}) =  \|x_{k}-x^*\|^2 + \frac{1}{\rho_{k}} |z_{k}|^2$ and $\mathcal{U}(x_k,z_k) = d(x_{k},\mathcal{X})^2 + \frac{1}{\rho_k} |z_{k}|^2$. With $c > 1$, we will always obtain a better convergence than the case with $c =1$. In that regard, one needs to obtain a reasonable lower bound of the quantity $\E[\|z_{k+1}\|^2_{G^{-1}}]$ in terms of $\mathcal{V}(x_k,z_k)$. Moreover, the choice $c= 1$ resolves into the case with fixed penalty $\rho$.
\end{remark}



\section{Conclusion}

In this work we proposed two variants of the Kaczmarz method for solving LS and LF problems. We exploited technical tools from continuous optimization to derive the RPK and RAK methods that extends the scope of the Kacmmarz method. We provided linear convergence results under mild conditions for the proposed methods. The proposed algorithms outperform the existing method on a wide variety of test instances. Moreover, the proposed work opens up the possibility of applying the developed techniques to the so-called \textit{Sketch \& Project} method. We rest that topic for future investigations.




\bibliographystyle{unsrt}

\bibliography{reference}

\end{document}